\documentclass[12pt, a4]{amsart}
\usepackage{amsgen,amsmath,amstext,amsbsy,amsopn,amsfonts,amssymb}
\usepackage{amsmath,amssymb,amsfonts,amsthm,enumerate}
\usepackage{mathrsfs}

\newtheorem{theorem}{Theorem}
\newtheorem{lemma}[theorem]{Lemma}

\newtheorem{corollary}[theorem]{Corollary}
\newtheorem{proposition}[theorem]{Proposition}
\newtheorem{obs}[theorem]{Observation} \newtheorem{defi}[theorem]{Definition}

\newtheorem{exa}[theorem]{Example}

\newtheorem{rem}[theorem]{Remark}
\newenvironment{remark}{\begin{rem}\rm}{\end{rem}}
\newtheorem{rems}[theorem]{Remarks}

\newtheorem{ack}[theorem]{Acknowlegment}

\def\P{\mathcal P}

\def\F{\mathcal F}

\def\B{\mathcal B}

\def\supp{{\rm supp}}
\def\NN{{\mathbf N}}
\def\ZZ{{\mathbf Z}}
\def\CCC{{\mathbf C}}
\def\RRR{{\mathbf R}}
\def\QQ{\mathbf Q}
\def\FF{\mathbf F}
\def\AA{{\mathbf A}}
\def\RR+{{\mathbf R}^*}
\def\TT{\mathbf T}
\def\KK{\mathbf K}
\def\kk{\mathbf k}

\def\PP{\mathbf P}
\def\Un{\mathbf 1}

\def\Q_p{{\mathbf Q}_p}
\def\SS{\mathbf S}

\def\Ga{\Gamma}
\def\ga{\gamma}

\def\Aut{{\rm Aut}}
\def\Aff{{\rm Aff}}

\def\paut{p_{\rm a}}

\def\supp{{\rm supp}}

\def\tout{\qquad\text{for all}\quad}

\begin{document}

\title[Spectral gap property for solenoids]{Spectral gap property and strong ergodicity for groups of affine transformations of solenoids}

%\address{Bachir Bekka \\ IRMAR \\ UMR-CNRS 6625 Universit\'e de  Rennes\\ Campus Beaulieu\\ F-35042  Rennes Cedex\\
% France}
%\email{bachir.bekka@univ-rennes1.fr}
%\author{Bachir Bekka and Camille Francini}
%\address{Camille Francini \\ IRMAR \\ UMR-CNRS 6625\\ ENS-Rennes\\
%Campus Beaulieu\\ F-35042  Rennes Cedex\\
% France}
%\email{camille.francini@ens-rennes.fr}
\address{Bachir Bekka \\ Univ Rennes \\ CNRS, IRMAR--UMR 6625\\
Campus Beaulieu\\ F-35042  Rennes Cedex\\
 France}
\email{bachir.bekka@univ-rennes1.fr}
\author{Bachir Bekka and Camille Francini}
\address{Camille Francini \\ Univ Rennes \\ CNRS, IRMAR--UMR 6625\\
Campus Beaulieu\\ F-35042  Rennes Cedex\\
 France}
\email{camille.francini@ens-rennes.fr}

%Univ Rennes,  CNRS, IRMAR - UMR 6625, F-35000 Rennes, France

\thanks{The authors acknowledge the support  by the ANR (French Agence Nationale de la Recherche)
through the projects Labex Lebesgue (ANR-11-LABX-0020-01) and GAMME (ANR-14-CE25-0004)}
\begin{abstract}
Let $X$ be a solenoid, that is, a compact finite dimensional connected  abelian group
with normalized Haar measure $\mu$ and let $\Ga\to \Aff(X)$ be an action of a countable discrete group 
$\Ga$   by  continuous affine transformations of $X$.
We show that the probability measure preserving  action $\Ga\curvearrowright (X,\mu)$ does not have the  spectral gap property
if and only if  there exists a  $\paut(\Ga)$-invariant  proper subsolenoid $Y$  of $X$   such that the image of $\Ga$ in $\Aff(X/Y)$ is a virtually solvable group,  where $\paut: \Aff(X)\to \Aut(X)$ is the canonical projection.
  When $\Ga$ is finitely generated or when $X$ is the $a$-adic solenoid for an integer $a\geq 1,$
the subsolenoid $Y$ can be chosen so that the image $\Ga$ in $\Aff(X/Y)$ is a virtually abelian group.
In particular,  an action $\Ga\curvearrowright  (X,\mu)$ by affine transformations on a solenoid $X$  has 
the  spectral gap property  if and only if  $\Ga\curvearrowright  (X,\mu)$ is strongly ergodic.

\end{abstract}

\maketitle
\section{Introduction}
\label{S0}
Let $X$ be a compact group and $\Aut(X)$  the group of continuous automorphisms
of $X$.  Denote by 
$$\Aff(X):=\Aut (X)\ltimes X$$
 the group of affine transformations of $X$, that is, of maps   of the form 
 $$X\to X, \quad x\mapsto x_0\theta(x)$$
for some $\theta\in \Aut(X)$ and $x_0\in X.$
Let $\mu$ be the normalized Haar measure of $X.$ By translation invariance and uniqueness of the Haar measure,  every  transformation in $\Aff(X)$   preserves $\mu$.

Given a group $\Ga$  and a homomorphism $\Ga\to \Aff(X),$ one has therefore  a measure  preserving action   $\Ga\curvearrowright  (X,\mu).$
The study of the ergodicity of such actions is a classical theme  going back to Halmos
\cite{Halmos} and Kaplansky \cite{Kaplansky}, both for the case where $\Ga=\ZZ$ is generated by a single automorphism of $X.$ 
For a characterization of the ergodicity of an action 
$\Ga\curvearrowright  (X,\mu)$ by automorphisms on an arbitrary compact group, see \cite[Lemma 2.2]{KitchensSchmidt}.
The following elementary proposition provides a neat characterization of the ergodicity  for actions by affine
transformations in the case where $X$ is moreover abelian and connected (for the proof, see Subsection~\ref{SS-Pro-Ergodic}
below).
\begin{proposition} 
\label{Pro-Ergodic}
Let $X$ be a compact connected abelian group and $\Ga\subset  \Aff(X)$  a  countable group.
Let $\paut: \Aff(X)\to \Aut(X)$ denote the canonical projection. The following properties are equivalent:
\begin{itemize} 
\item[(i)] The action $\Ga\curvearrowright  (X,\mu)$ is not ergodic. 
 \item[(ii)] There exists a $\paut(\Ga)$-invariant  \emph{proper} and connected closed subgroup $Y$  
of $X$  such that the image of $\Ga$ in $\Aff(X/Y)$ is a finite group.
 \end{itemize}
\end{proposition}

Our main concern in this article is the spectral gap property for the action $\Ga\curvearrowright  (X,\mu)$.
Let $\pi_X$ denote the corresponding Koopman   representation of $\Ga$ on $L^2(X,\mu)$.
Recall that  $\Ga\curvearrowright  (X,\mu)$
is ergodic  if and only if there is  no non-zero invariant vector in the $\pi_X(\Ga)$-invariant subspace
$L^2_0(X,\mu)=(\CCC\Un_X)^\perp$ of functions with zero mean.
The action  $\Ga\curvearrowright  (X,\mu)$ is said to have the \textbf{spectral gap property}
(or has a spectral gap) if  there are  not even  almost invariant vectors in $L^2_0(X,\mu),$ 
that is, there is no sequence of unit vectors
 $f_n$  in  $ L^2_0(X,\mu)$  such
 that $\lim_n\Vert \pi_X(\ga)f_n-f_n\Vert=0$ for all $\ga\in \Ga.$
 
Group actions on general probability spaces with the spectral gap property have an amazing range of applications ranging  from geometry and group theory to operator algebras and graph theory  (for an account on this property, see \cite{BekkaSG}).

Given a specific non abelian compact group $X$, there is in general no known characterization of
the countable subgroups $\Ga$ of $\Aff(X)$ such that $\Ga\curvearrowright  (X,\mu)$   has the spectral gap property:
indeed, it is usually a difficult problem to even find subgroups $\Ga$ of $X$ for which the action $\Ga$
by translations on $X$ has a spectral gap (for a recent result  in the case $X=SU(d)$, see  \cite{BourgainGamburd}).

We characterize below (Theorem~\ref{Theo1})  actions  by affine transformations $\Ga\curvearrowright  (X,\mu)$ with the spectral gap property for a solenoid  $X$, in the same spirit as the ergodicity characterization from Proposition~\ref{Pro-Ergodic}. This result (as well as Theorem~\ref{Theo2}\  below)  generalizes Theorem~5 in \cite{BachirYves},
where an analogous characterization was given for the case of a torus $X=\TT^d$  (see also \cite[Theorem 6.5]{FurmanShalom} for a partial result).

Recall that a \textbf{solenoid} $X$ is a finite-dimensional, connected, compact abelian group
(see \cite[\S 25]{HeRo--63}).
Examples of solenoids of dimension $d\geq 1$ include the  torus $\TT^d=\RRR^d/\ZZ^d$ as well as  the 
$p$-adic solenoid $\SS_p^d$ for where $p$ is  a prime integer (see Appendix to Chapter I in \cite{Robe--00})  or, more generally, the $a$-adic solenoid $\SS_a ^d$ for a positive integer $a$ (see below).
In some sense the largest $d$-dimensional solenoid is provided by the  solenoid $\AA^d/\QQ^d$, where $\AA$ is the ring of ad\`eles over $\QQ$ (see Subsection~\ref{SS-Dual}).

Here is our main result. Recall that, given a group property $\P,$  a group has  virtually $\P$ if it has a finite index subgroup with the property $\P.$ A subsolenoid of a solenoid $X$ is a closed and connected subgroup of $X.$

\begin{theorem}
\label{Theo1}
Let $X$ be a solenoid with normalized Haar measure $\mu$ and $\Ga$ a countable subgroup of $\Aff(X)$.
Let $\paut: \Aff(X)\to \Aut(X)$ denote the canonical projection.
 The following properties are equivalent:
\begin{itemize}
 \item [(i)] The action $\Ga\curvearrowright  (X,\mu)$  does not have  the spectral gap property.
 \item [(ii)] The action $\paut(\Ga)\curvearrowright  (X,\mu)$  does not have  the spectral gap property.
\item [(iii)]  There exists a $\paut(\Ga)$-invariant  \emph{proper subsolenoid} $Y$   of $X$  such that the image of $\Ga$ in $\Aff(X/Y)$ is an amenable group.
\item [(iv)] There exists a $\paut(\Ga)$-invariant  \emph{proper subsolenoid} $Y$   of $X$  such that the image of $\Ga$ in $\Aff(X/Y)$ is  a virtually solvable group.
\end{itemize}
\end{theorem}

The proof of Theorem~\ref{Theo1} is  an extension to the adelic setting of the  methods from \cite{BachirYves} and is based
  on the consideration of appropriate invariant means on  finite dimensional vector spaces over local fields
  and the associated invariant measures on the corresponding projective spaces.
 
 \begin{remark}
 \label{Rem1}
 Theorem~\ref{Theo1} can be sharpened in the case where  $\Ga$ is a \emph{finitely generated}  subgroup $\Aff(X)$:
  the subsolenoid $Y$ in (iv) can be chosen so  that the image of $\paut(\Ga)$ in $\Aut(X/Y)$ is  virtually  abelian (see Remark~\ref{Rem3}).
 \end{remark}

The spectral gap property is related to  another strengthening
of ergodicity. Recall that the action of a  countable group $\Ga$
 by measure preserving transformations on a probability space $(X, \mu)$ is  \textbf{strongly ergodic} 
 (see \cite{Schmidt}) if every
 sequence $(A_n)_n$ of measurable subsets of  $X$ which is asymptotically invariant
 (that is,  which is such that  $\lim_n\mu(\ga A_n \bigtriangleup A_n)=0$ for all $\ga\in \Ga$)
is trivial (that is, $\lim_n\mu( A_n)(1-\mu(A_n))=0$).
As is easily seen,  the spectral gap property implies strong ergodicity (the converse 
implication does not hold in general; see Example 2.7 in  \cite{Schmidt}).
Moreover, no ergodic measure preserving action of an \emph{amenable} group on 
a  non atomic probability space is strongly ergodic, by   \cite[Theorem 2.4]{Schmidt}.
The following corollary  is therefore a direct consequence of Theorem~\ref{Theo1}.
\begin{corollary}
\label{Cor1}
Let $X$ be a solenoid and $\Ga\subset  \Aff(X)$  a  countable group.
 The following properties are equivalent:
\begin{itemize} 
\item[(i)] The action $\Ga\curvearrowright  (X,\mu)$
 has the spectral gap property. 
 \item[(ii)]  The action $\Ga\curvearrowright  (X,\mu)$ is strongly ergodic. 
 \end{itemize}
\end{corollary}
It is worth mentioning that the equivalence of (i) and (ii) in Corollary~\ref{Cor1} holds for actions by \emph{translations} 
on a general compact group $X$ (see Proposition~3.1 in \cite{Abert}).

We can prove improve   Theorem~\ref{Theo1} in the case of $a$-adic solenoids. Let $a$ be a square free positive integer, that is, $a=p_1\cdots p_r$ is a product of different primes $p_i.$
Then
$$ \AA_a^d:= \RRR^d\times\QQ_{p_1}^d\times\cdots\times \QQ_{p_r}^d,$$
is a locally compact ring, where $\QQ_p$ is the field of $p$-adic numbers.
Let $\ZZ[1/a]=\ZZ[1/p_1, \cdots, 1/p_r]$ denote the subring  of $\QQ$ generated by $1$ and $1/a$. 
Through the diagonal embedding 
$$
\ZZ[1/a]^d \to \AA_a^d, \qquad b\mapsto (b, b,\cdots, b),
$$
we may identify $\ZZ[1/a]^d$ with a  discrete and cocompact subring of $\AA_a^d.$
The  \textbf{$a$-adic solenoid} is defined as the quotient 
$$
\SS_a= \AA_a^d/ \ZZ[1/a]^d.
$$
(see Chap. II, \S 10 in \cite{HeRo--63}).
Moreover, $\Aut(\AA_a^d)$ is canonically isomorphic to $GL_d( \RRR) \times GL_d(\QQ_{p_1}) \times\cdots\times GL_d(\QQ_{p_r})$  and so  $\Aut(\SS_a^d)$ can be identified with   $GL_d(\ZZ[1/a])$.

For a subset $S$ of $GL_d(\KK)$ for a field $\KK,$ we denote by $S^t=\{g^t \mid g\in S\}$ the set 
of transposed matrices from $S$.
\begin{theorem}
\label{Theo2}
Let $a\geq 1$ be a square free integer.
Let $\Ga$ be a subgroup of $\Aff(\SS_a^d)\cong GL_d(\ZZ[1/a])\ltimes \SS_a^d.$
The following properties are equivalent:
\begin{itemize}
 \item [(i)] The action $\Ga\curvearrowright  (\SS_a^d,\mu)$  does not have  the spectral gap property.
 \item [(ii)]  There exists a  non zero linear subspace $W$  of $\QQ^d$ which is invariant under $\paut(\Ga)^t\subset GL_d(\QQ)$ and such that the image of $\paut(\Ga)^t$ in $GL(W)$ is a virtually abelian group.
\end{itemize}
\end{theorem}
Examples of group actions  on  solenoids  with the spectral gap property are provided by the following immediate consequence of Theorem~\ref{Theo2}.
\begin{corollary}
\label{Cor2}
For a square free integer $a\geq 1,$ let $\Ga$ be subgroup of $GL_d(\ZZ[1/a])$. Assume that $\Ga$ is not virtually abelian  and that  $\Ga$ acts irreducibly on $\QQ^d$.
 Then the action of $\Ga$ by automorphisms of $\SS_a^d$ has the spectral gap property.
\end{corollary}

\begin{remark} 
\label{Rem2}
Corollary~\ref{Cor2} generalizes Theorem~6.8 in \cite{FurmanShalom}  in which the same result is proved under the stronger assumption that $\Ga$ acts irreducibly on $\RRR^d$.
\end{remark}

 This paper is organized as follows. In Section 1, we establish and recall some  preliminary facts
 which are necessary to the proofs of our results. Section 2 is devoted to the proofs  of Theorem~\ref{Theo1}, Theorem~\ref{Theo2}, and Proposition~\ref{Pro-Ergodic}.  
 
 \medskip
 \noindent
\textbf{Acknowlegments} 
We are grateful to E.~Breuillard and Y. Guivarc'h for useful discussions related to the main theme of this article.
Thanks are also due to   V.~Guirardel   who suggested  the proof of Lemma~\ref{Lem-LinearVirtually}.

\section{Some preliminary results}
\label{S1}
\subsection{Reduction to the case of automorphisms}
\label{SS-Reduction}
Let $X$ be a  compact abelian group with normalized Haar measure $\mu$
and $\Ga$ a countable subgroup of $\Aff(X)$. 
The aim of this subsection is to reduce the study of the spectral gap property for 
 $\Ga\curvearrowright  (X,\mu)$ to that of the action  $\paut(\Ga)\curvearrowright  (X,\mu)$,
where $\paut: \Aff(X)\to \Aut(X)$ is the canonical projection.

Let $\widehat{X}$ be the Pontrjagin dual group of $X$, which is a discrete group.
The group $\Aut(X)$ acts by duality on  $\widehat{X}$, giving rise to a \emph{right} action
$\widehat{X}\times \Aut(X)\to \widehat{X}$ given by 
$$
\chi^\theta(x)= \chi(\theta(x)) \tout \theta\in \Aut(X), \chi\in \widehat{X}, x\in X.
$$
The Fourier transform $\F: L^2(X,\mu)\to \ell^2(\widehat{X})$, given by 
$$(\F f)(\chi)= \int_X f(x) \overline{\chi}(x)d\mu(x) \tout f\in L^2(X,\mu), \chi\in \widehat{X},$$ is a Hilbert space isomorphism.
The Koopman representation of $\Aff(X)$ on $L^2(X,\mu)$ corresponds under $\F$ to the unitary representation 
$\pi_X$ of $\Aff(X)$ on $\ell^2(\widehat X)$ given by 
$$
\pi_X(\ga)(\xi)(\chi)= \chi(x)\xi(\chi^{\theta})  \tout \xi\in \ell^2(\widehat X) , \chi\in \widehat{X}, \leqno{(*)}
$$
for  $\ga=(\theta, x)$ in  $\Aff(X)= \Aut(X)\ltimes X.$ 
 Observe that $L^2_0(X,\mu)$ corresponds under $\F$ to the subspace $\ell^2(\widehat{X}\setminus \{\Un_X\})$ of 
 $\ell^2(\widehat{X}).$ 
 
\begin{proposition}
\label{Pro-AnalysisKoopmanAffine}
Let $X$ be a compact  abelian group with normalized Haar measure $\mu$
and let $\Ga$ be a countable subgroup of $\Aff(X)$ such that the action $\Ga\curvearrowright  (X,\mu)$  does not have  the spectral gap property.
Then the action $\paut(\Ga)\curvearrowright  (X,\mu)$  does not have  the spectral gap property.
\end{proposition}
\begin{proof} 
We realize the Koopman representation $\pi_X$  on $\ell^2(\widehat{X})$ as above. 
Since $\Ga\curvearrowright  (X,\mu)$ does not have  the spectral gap property,  there exists a sequence $(\xi_n)_{n\geq 1}$ of unit vectors in $\ell^2(\widehat{X}\setminus \{\Un_X\})$ such that  $\lim_n\Vert  \pi_X(\ga)\xi_n-\xi_n\Vert=0$, that is, by Formula $(*),$
$$
\lim_n \sum_{\chi\in \widehat{X}} |\chi(x) \xi_n(\chi^\theta)- \xi_n(\chi)|^2 =0,
$$
for every $\ga=(\theta, x)\in \Ga.$

For $n\geq 1$, set $\eta_n=|\xi_n|.$ 
Then $\eta_n$ is a unit vector in $\ell^2(\widehat{X}\setminus \{\Un_X\})$
and, for  every $\ga=(\theta, x)\in \Ga,$ we have
$$
\begin{aligned}
\Vert \pi_X(\theta)\eta_n- \eta_n\Vert^2&= 
%\sum_{\chi\in \widehat{X}} |\eta_n(\chi^\theta)- \eta_n(\chi)|^2\\
%&= \sum_{\chi\in \widehat{X}}\left \vert |\xi_n(\chi^\theta)|- |\xi_n(\chi)|\right\vert^2\\
\sum_{\chi\in \widehat{X}}\left \vert |\xi_n(\chi^\theta)|- |\xi_n(\chi)|\right\vert^2\\
&= \sum_{\chi\in \widehat{X}}\left \vert |\chi(x)\xi_n(\chi^\theta)|- |\xi_n(\chi)|\right\vert^2\\
&\leq  \sum_{\chi\in \widehat{X}} |\chi(x)\xi_n(\chi^\theta)- \xi_n(\chi)|^2\\
&=\Vert \pi_X(\ga)\xi_n- \xi_n\Vert^2.
\end{aligned}
$$
Hence, $(\eta_n)$ is a sequence of almost $\pi_X(\paut(\Ga))$-invariant and so   $\paut(\Ga)\curvearrowright  (X,\mu)$  does not have  the spectral gap property.
 
\end{proof}

\subsection{Invariant means, invariant measures, and  linear actions}
\label{SS-Means}
Let $X$ be a locally compact topological space. A \textbf{mean} on $X$ is  positive linear functional $M$ on  the space $C^b(X)$ of continuous bounded functions on $X$  such that  $M({\Un}_X)=1.$
If $Y$ is another locally compact topological space and $\Phi: X\to Y$ a continuous map,
the pushforward $\Phi_*(M)$  of $M$ by $\Phi$  is the mean   on $Y$ given by 
$\Phi_*(M)(f)= M(f\circ \Phi)$ for $f\in C^b(Y).$

Let $\Ga$ be a group and  $\Ga\curvearrowright X$  an action of 
 $\Ga$ by homeomorphisms of $X$.
A $\Ga$-invariant mean on $X$ is a mean on $M$ which is invariant for the 
induced action of $\Ga$ on $C^b(X).$
The following  lemma is well-known and easy to prove.
\begin{lemma}
\label{Lem1}
Let $X,Y$ be respectively  a locally compact  space and a compact space.
Let  $\Ga\curvearrowright X$ and  $\Ga\curvearrowright Y$  be actions of 
the group $\Ga$ by homeomorphisms of $X$ and $Y$. 
Let  $\Phi: X\to Y$ a continuous $\Ga$-equivariant map. 
Assume that there exists $M$ be an invariant mean
on $X$. Then $\Phi_*(M)$ is given by integration against a $\Ga$-invariant probability measure $\mu$ on $Y$.
\end{lemma}
\begin{proof} 
Since $\Phi_*(M)$ is a positive linear functional on $C(Y)$ and since $Y$ is compact, there exists
 by the Riesz representation theorem a probability measure $\mu$ on $Y$
such that 
$$\int_{Y} f(x) d\mu(x) =\Phi_*(M)(f)  \tout f\in C(Y).$$
The measure $\mu$ is $\Ga$-invariant, since $\Phi_*(M)$ is $\Ga$-invariant.
\end{proof}    

Let $\kk$ be a local field (that is, a non discrete locally compact field)  and $V$ a finite dimensional vector space over $\kk$.
 Then $V$ is a locally compact vector space and  $GL(V)$ is a locally compact group,  for the topology inherited from $\kk$.
This is the only topology  on $GL(V)$ we will consider in the sequel (with the exception of the proof of Lemma~\ref{Lem-LinearVirtually}).

Every subgroup $\Ga$ of $GL(V)$ acts by homeomorphisms on the projective space $\PP(V).$
A crucial tool for our proof of Theorems~\ref{Theo1} and \ref{Theo2} is the consideration 
of   $\Ga$-invariant probability measures on $\PP(V),$ a theme which goes back to the proof of the Borel density theorem in  \cite{Furstenberg}.  The following proposition summarizes the main  consequences, as we will use them,  of the existence of such a measure. Variants of this proposition  appeared already  at several places (see for instance \cite{BachirYves},  \cite{Cornulier},  
\cite{FurmanShalom}), but not exactly in the form we need; so,  we will   briefly indicate  its proof. 

For a group $G,$ we denote by $[G,G]$ the commutator subgroup of $G.$
\begin{proposition}
\label{Pro-Furstenberg}
Let $V$ be a finite dimensional vector space over  a local field $\kk$ and
$G$  a closed  subgroup of $GL(V)$.
Assume that there exists a $G$-invariant probability measure  on the Borel subsets of  $\PP(V)$
which is not supported on a proper projective subspace. Then there exists  a subgroup $G_0$ of $G$ of finite index such 
that $[G_0, G_0]$ is relatively compact  in $GL(V)$.  In particular,  the locally compact group $G$ is amenable.  
\end{proposition} 
\begin{proof} 
Let $\nu$ be $G$-invariant probability Borel measure on $\PP(V)$.
%and denote by $\pi: V\to \PP(V)$  the canonical projection. 
As in the proof  of  Lemma 11 in \cite{BachirYves} or of Theorem 6.5.i in \cite{FurmanShalom}, there exists
finitely many subspaces $V_1,\dots, V_r$ of $V$ and a subgroup $G_0$ of finite index in $G$ with the following properties:
\begin{itemize}
\item $\nu$ is supported by the union of the projective subspaces corresponding to the $V_i$'s; 
\item $G_0$ stabilizes  $V_i$  for every $i\in \{1,\dots, r\};$
\item  the image of $G_0$ in $PGL(V_i)$ is relatively compact for every $i\in \{1,\dots, r\}.$
\end{itemize}
Since the  image of  the commutator subgroup $[G_0, G_0]$ in $GL(V_i)$ is contained $SL(V_i)$,
it follows that the  image of  $[G_0, G_0]$ in $GL(V_i)$ is relatively compact for every $i\in \{1,\dots, r\}.$
As $\nu$ is not supported  on a proper projective subspace, the linear span of $V_1\cup \dots \cup V_r$
coincides with $V.$ 
This implies that $[G_0, G_0]$ is relatively compact in $GL(V).$ Therefore, $G_0$ and hence $G$ is amenable.
\end{proof}

 A further ingredient we will need  is the following result  which is Proposition 9 and Lemma 10 in \cite{BachirYves}; observe that, although only the case $\kk=\RRR$ is considered there,  the
 arguments for the proof  are valid without  change for any  local  field $\kk.$
 \begin{proposition}
 \label{Pro-AmenableImage}
 Let $V$ be a finite dimensional vector space over  a local field $\kk$ and $G$  a subgroup of $GL(V).$ 
  There exists a largest $G$-invariant linear subspace $V(G)$ of $V$
 such that the closure of the image of $G$ in $GL(V(G))$ is an amenable locally compact group. 
 Moreover, we have $\overline{V}(G)=\{0\}$ for   $\overline{V}= V/V(G)$.
 $\square$
 \end{proposition}
 
We will also need the following  (probably well-known) lemma, for which we could not find a reference.
Recall that a  group is \emph{linear}  if it can be embedded as a subgroup of $GL_n(\kk)$ for some field $\kk.$
\begin{lemma}
\label{Lem-LinearVirtually}
Let $\Ga$ be a linear group. 
Assume that $\Ga$ is finite-by-abelian (that is, $\Ga$ is a finite extension of an abelian group).
Then $\Ga$ is virtually abelian (that is, $\Ga$ is abelian-by-finite).
\end{lemma}
\begin{proof}
We may assume that  $\Ga$ is a subgroup of $GL_n(\kk)$ for an algebraically closed field $\kk.$

By assumption, there exists  a finite normal subgroup of $\Ga$
containing  $[\Ga, \Ga].$ In particular, $[\Ga, \Ga]$ is finite.

Let $G\subset GL_n(\kk)$ be the closure of $\Ga$ in the Zariski topology.
Since $[\Ga, \Ga]$ is finite, $[\Ga, \Ga]$ is a Zariski closed subgroup of $G.$ 
It follows that $[G,G]= [\Ga,\Ga]$ and hence that $[G,G]$ is finite. In particular, $[G^0, G^0]$ is finite,
where $G^0$ is the Zariski connected component  of $G.$ 
However, $[G^0, G^0]$ is connected (see Proposition 17.2 in \cite{Humphreys}).
Therefore, $[G^0, G^0]=\{e\},$ that is, $G^0$ is abelian.
Let $\Ga^0=\Ga \cap G^0$. Then $\Ga^0$ is a subgroup of finite index in $\Ga$
and $\Ga^0$ is abelian. 
\end{proof}
Observe that the previous lemma does not hold for non linear groups: 
 let $V$ be an infinite dimensional vector space over a finite field $\FF$ of characteristic different from $2$
  and $\omega: V \times V \rightarrow \FF$ a symplectic form on $V.$ 
 Let $\Ga$ be the associated  ``Heisenberg group", that is,    $\Ga= V \times \FF$ 
 with the law $(v, \lambda) (w,\beta) = (v+w, \lambda+\beta+\omega(v,w))$. Then $\Ga$ is finite-by-abelian but not virtually abelian. 
 
\subsection{The dual group of a solenoid and the ring of ad\`eles}
\label{SS-Dual}
Solenoids are characterized in terms of  their Pontrjagin dual groups as follows. 
Recall that the rank (also called Pr\"ufer rank) of an abelian group $A$ is the cardinality of a maximal linearly independent subset of $A$.
A compact abelian group $X$ is a  solenoid 
if and only if  $\widehat{X}$  is a finite rank, torsion-free, abelian group;
when this is the case,  the topological dimension of $X$ coincides with the rank of $\widehat{X}$
(see Theorem (23.18) in \cite{HeRo--63}). 

Let $X$ be a solenoid. 
  Let $d\geq 1$ be the rank of $\widehat{X}.$ Since  $\widehat{X}$ is torsion-free,
  $$V_\QQ:= \widehat{X} \otimes_\ZZ \QQ$$
is a $\QQ$-vector space of dimension $d$ and we may (and will) view $\widehat{X}$ as a subgroup of $V_\QQ$
via the embedding
$$
\widehat{X} \to   V_\QQ, \qquad \chi\mapsto \chi\otimes 1.
$$
(Since, obviously, every subgroup of $\QQ^d$ is  torsion-free abelian group of finite rank, we see  that the solenoids are exactly the  dual groups of subgroups of $\QQ^d$ for some $d\geq 1.$)

We will need to further embed $\widehat{X}$ in  vector spaces over various local fields. 
Let $\P$ be the set of primes of $\NN.$ Recall that,
for every  $p\in \P ,$ the additive group of the field $\QQ_p$ of $p$-adic numbers is a locally compact  group
containing the subring $\ZZ_p$ of $p$-adic integers as compact open subgroup.
The ring $\AA$ of ad\`eles of $\QQ$  is the restricted product  $\AA= \RRR\times \prod_{p\in \P} (\QQ_p, \ZZ_p)$
relative to the subgroups $\ZZ_p$; thus, 
$$\AA= \left\{(a_\infty, a_2, a_3, \cdots) \in \RRR\times \prod_{p\in \P} \QQ_p\mid a_p \in \ZZ_p \text{ for almost every } p\in \P \right\}.$$
The field $\QQ$ can be viewed as discrete and cocompact subring of the locally compact 
ring $\AA$ via the diagonal embedding
$$
\QQ\to \AA, \qquad q\mapsto (q, q, \dots).
$$
Set $\QQ_\infty := \RRR$ and for  $p\in \P\cup \{\infty\},$  set  
$$V_{p}=V_\QQ \otimes_\QQ \QQ_p.$$ 
 Then $V_p$ is a   $d$-dimensional vector space over $\QQ_p$  and
$V_\QQ$ can be viewed as a subspace of $V_p$ for every  $p\in \P\cup \{\infty\}.$

Fix  a basis $\B$ of $V_\QQ$ over $\QQ$ contained in $\widehat{X}.$  Then $\B$ is a basis of $V_p$
over $\QQ_p$ for every $p\in \P\cup \{\infty\}.$
For $p\in \P,$ let $\B_p$ be the $\ZZ_p$-module generated by 
$\B$ in $V_p$.  The restricted product  
$$V_\AA= V_{\infty} \times \prod_{p\in \P} (V_p, \B_p)$$
is a locally compact ring and $V_\QQ$ embeds diagonally 
 as a discrete and cocompact subgroup of $V_\AA$ 
 (for all this,  see Chap. IV in \cite{Weil}). 

As a result of this discussion,  we can view $\widehat{X}$ as a subgroup of 
$V_\QQ$ which is itself a discrete and cocompact subgroup of $V_\AA$.
Since the dual group of $V_\QQ\cong  \QQ^d$ may be identified 
with  $\AA^d/ \QQ^d$ (see Subsection~\ref{ProofTheo2}), observe that $X$  is a quotient of the full $d$-dimensional solenoid
$\AA^d/ \QQ^d$.

We discuss now the automorphisms of $\widehat{X}.$ 
Every $\theta \in \Aut(\widehat{X})$ extends, in a unique way, to an automorphism 
$\widetilde{\theta}$ of $V_\QQ$ defined by 
$$
\widetilde{\theta}( \chi\otimes n/m)= \theta(n\chi)\otimes (1/m) \tout \chi\in \widehat{X}, \, n/m\in \QQ.
$$
Therefore, we may identify $\Aut(\widehat{X})$ with a subgroup $GL(V_\QQ)$.
So, $\Aut(\widehat{X})$ embeds diagonally as a \emph{discrete} subgroup of the locally compact group $GL(V_\AA)\cong GL_d(\AA)$, which  is also  the restricted product
$$GL(V_\AA)=  GL(V_{\infty}) \times \prod_{p\in \P} (GL(V_p), GL(\B_p)).$$

Let $p_1,\cdots, p_r$ be   different primes and $a=p_1\dots p_2.$ Let 
$\Aut(\widehat{X})_{\ZZ[1/a]}$ be  the subgroup of all $\theta \in \Aut(\widehat{X})$
such that 
$$
\theta(\B_p)=\B_p \tout p\in \P \setminus \{p_1, \cdots, p_r\}.
$$ 
Then $\Aut(\widehat{X})_{\ZZ[1/a]}$ may be identified with a subgroup of $GL_d(\ZZ[1/a])$ and
embeds  diagonally as a \emph{discrete} subgroup of the locally compact group
$$
GL(V_{\infty}) \times GL(V_{p_1}) \times\cdots\times GL(V_{p_r}).
$$

\subsection{Dual of the $a$-adic solenoid}
\label{SS-DualSol}
Set $X:= \SS_a^d$ for a square free integer $a=p_1\dots p_r.$ 
Recall that $X= \AA_a^d/ \ZZ[1/a]^d,$ with $\ZZ[1/a]^d$ diagonally embedded in  
the locally compact ring
$$
\AA_a^d:= \RRR^d\times\QQ_{p_1}^d\times\cdots\times \QQ_{p_r}^d.
$$
We identify $\widehat{\RRR}$ with $\RRR$ via the map 
 $\RRR\to \widehat{\RRR}, y\mapsto e_y$ given by 
 $e_y(x)= e^{2\pi i xy}$  and $\widehat{\QQ_p}$ with 
$\QQ_p$  via the map  $\QQ_p\to \widehat{\QQ_p}, y\mapsto \chi_y$ given by 
 $\chi_y(x)= \exp(2\pi i  \{xy\}),$ where $\{x\}= \sum_{j=m}^{-1} a_j p^j $ denotes the ``fractional part" of a 
 $p$-adic number $x= \sum_{j=m}^\infty a_j p^j$ for integers $m\in \ZZ$ and $a_j \in \{0, \dots, p-1\}$
 (see \cite[Section D.4]{BHV}).
 
Then  $\widehat{\AA_a^d}$ is identified  with $\AA_a^d$ and 
$\widehat{X}$  with the annihilator of $\ZZ[1/a]^d$ in
$\AA_a^d,$ that is, with $\ZZ[1/a]^d$ embedded in 
$\AA_a^d$ via the map
$$
\ZZ[1/a]^d \to \AA_a^d, \qquad b\mapsto (b, -b \cdots,- b).
$$
Under this identification, the dual action of the automorphism group 
$$
\Aut(\AA_d^d)\cong  GL_d(\RRR)\times GL(\QQ_{p_1})\times\cdots\times GL(\QQ_{p_r})
$$
on  $\widehat{\AA_a^d}$  corresponds to the right action on  $\RRR^d\times\QQ_{p_1}^d\times\cdots\times \QQ_{p_r}^d$ given by 
$$
((g_\infty, g_1, \cdots, g_r), (a_\infty, a_1, \cdots, a_r))\mapsto  (g_\infty^t a_\infty , g_1^t a_1, \cdots, g_r^t a_r),
$$
where $(g,a)\mapsto ga$ is the usual (left) linear action of $GL_d(\kk)$ on $\kk^d$ for a field $\kk.$

\section{Proofs of Theorem~\ref{Theo1}, Theorem~\ref{Theo2}, and Proposition~\ref{Pro-Ergodic}}

\subsection{Proof of Theorem~\ref{Theo1}}
\label{ProofTheo1}
Proposition~\ref{Pro-AnalysisKoopmanAffine} shows that (i) implies (ii).
The fact that (iii) implies (i) follows a general result :
a measure preserving action  of a countable  amenable group on a non atomic probability $(Y,\nu)$ never has the  spectral gap property  (see \cite[Theorem 2.4]{JuRo} or \cite[(2.4) Theorem]{Schmidt}).
Since $\Ga$, which is isomorphic to  a subgroup of $GL_d(\QQ)$,
is a linear subgroup over a field of characteristic zero, (iii) implies (iv) by  one part of Tits' alternative theorem (\cite{Tits}).
As  (iii) is an obvious consequence of  (iv),  it remains  to show that (ii) implies (iii).
We will proceed  in several steps.

\vskip.2cm
$\bullet$ {\it First step.}    Assume that there exists $\paut(\Ga)$-invariant  proper subsolenoid $Y$   of $X$  such that the image $\Delta$ of  $\paut(\Ga)$ in $\Aut(X/Y)$ is amenable.
We claim that  the image of $\Ga$ in $\Aff(X/Y)$ is amenable.

\par
Indeed,  the image of $\Ga$ in $\Aff(X/Y)$ is a subgroup of 
$\Delta\ltimes (X/Y)$. Since $X/Y$ is abelian, $\Delta\ltimes (X/Y)$ is amenable (as discrete group) and the claim follows.

\vskip.2cm
In view of the first step, we may and will assume in the sequel that $\Ga\subset \Aut(X).$
By duality, we can also view $\Ga$ as a subgroup of $\Aut(\widehat{X}).$
In the sequel, we write  $0$ for the neutral element in $\widehat{X}$ instead of $\Un_X.$

\vskip.2cm
$\bullet$ {\it Second step.} We claim that there exists a $\Ga$-invariant mean  on $\widehat{X}\setminus \{ 0\}.$

\par

Indeed, since $\Ga\curvearrowright  (X,\mu)$  does not have a spectral gap, this follows by  a standard argument:
there  exists  a sequence $(\xi_n)_{n\geq 1}$ of unit vectors in $\ell^2(\widehat{X}\setminus \{ 0_X\})$ such that 
$$
\lim_n \Vert  \pi_X(\ga)\xi_n- \xi_n\Vert_2=0 \tout \ga\in \Ga,
$$
for the associated Koopman representation $\pi_X$ (see  Proof of Proposition~\ref{Pro-AnalysisKoopmanAffine}).
 Then $\eta_n:=|\xi_n|^2$ is a unit vector in $\ell^1(\widehat{X}\setminus \{0\})$  and 
$$
\lim_n \Vert  \pi_X(\ga)\eta_n- \eta_n\Vert_1=0 \tout \ga\in \Ga.
$$
Any  weak$^*$-limit  of $(\eta_n)_n$ in the dual space of $\ell^\infty(\widehat{X}\setminus \{0\})$ is a $\Ga$-invariant mean on $\widehat{X}\setminus \{0\}.$

\vskip.2cm
Let $d$ be the rank of $\widehat{X}.$ As in Subsection~\ref{SS-Dual}, we embed $\widehat{X}$ in the $d$-dimensional $\QQ$-vector space $V_\QQ= \widehat{X} \otimes_\ZZ \QQ$ as well as in the $d$-dimensional $\QQ_p$-vector spaces $V_{p}=V_\QQ\otimes_\QQ \QQ_p$ for $p\in \P \cup \{\infty\}$, where $\P$ is the set of primes and where $\QQ_\infty=\RRR$.
Accordingly, we identify $\Aut(\widehat{X})$ with a subgroup of $GL(V_\QQ)$.

\vskip.2cm 
We fix  a $\Ga$-invariant mean $M$ on $\widehat{X}\setminus \{0\},$ which we view as
mean on $\widehat{X}$ and write $M(A)$ instead of $M(\Un_A)$ for a subset 
$A$ of $\widehat{X}.$

\vskip.2cm 
$\bullet$ {\it Third step.}   Let $p\in \P \cup \{\infty\}.$  We claim  that 
$$M(\widehat{X}\cap V_p(\Ga)) =1,$$
 where $V_p(\Ga)$ is the $\Ga$-invariant linear subspace of $V_p$  defined  in  Proposition~\ref{Pro-AmenableImage}.

The proof of the claim is similar to the proof of Proposition 13 in \cite{BachirYves}; for the convenience of the reader,
we repeat  the main arguments. 
Assume, by contradiction, that $M(\widehat{X}\cap V_p(\Ga))<1.$
We therefore have
$$
t: =M(\widehat{X}\setminus V_p(\Ga))>0.
$$
Then  a $\Ga$-invariant mean $M_1$ is defined on $\widehat{X}\setminus V_p(\Ga)$
by 
$$
M_1(A)= \dfrac{1}{t} M(A) \tout A\subset \widehat{X}\setminus V_p(\Ga).
$$
Consider the quotient vector space   $\overline{V_p}= V_p/V_p(\Ga)$
with the induced $\Ga$-action.
The image of  $\widehat{X}\setminus V_p(\Ga)$ 
under the canonical projection $j: V_p \to \overline{V_p}$ does not contain $\{0\}.$ So, $\overline{M_1}:=j_*(M_1)$ is a $\Ga$-invariant mean on $\overline{V_p}\setminus\{0\}.$ 
By Lemma~\ref{Lem1}, the pushforward of $M_1$ on the projective $\PP(\overline{V_p})$  defines a $\Ga$-invariant probability measure $\nu$ on $\PP(\overline{V_p}).$

Let $\overline{W}$ be the  linear span of  the inverse image of $\supp(\nu)$ in $\overline{V_p}.$
Then $\overline{W}\neq \{0\}$ and  $\nu$ is not supported on a proper
projective subspace in $\PP(\overline{W}).$ Proposition~\ref{Pro-Furstenberg} shows that 
the closure of the image of $\Ga$ in $GL(\overline{W})$ is an amenable group. 
It follows that  $\overline{V_p}(\Ga)\neq \{0\}.$
This contradicts Proposition~\ref{Pro-AmenableImage}.

\vskip.2cm
Let $\P=\{p_1, p_2, p_3, \cdots \}$ be an enumeration of the set $\P$ of prime integers.
\vskip.2cm 
$\bullet$ {\it Fourth step.}   We claim that, for every $n\in  \NN,$ we have
 $$\widehat{X}\cap V_\infty(\Ga)\cap \bigcap_{i=1}^n V_{p_i}(\Ga) \neq \{0\}.$$
 
Indeed, by the third step, we have $M(\widehat{X}\setminus\{0\} \cap V_{p}(\Ga))=1$ for every 
$p\in \{p_1,\dots, p_r\} \cup \{\infty\}.$ By finite-additivity of $M,$ it follows that 
$$
M\left(\widehat{X}\setminus\{0\} \cap  V_\infty(\Ga)\cap \bigcap_{i=1}^n V_{p_i}(\Ga)\right)  =1;
$$
this proves the claim in particular.

\vskip.2cm
Fixing  a basis $\B$ of $V_\QQ$ over $\QQ$ contained in $\widehat{X}$,
and denoting by  $\B_p$  the $\ZZ_p$-module generated by 
$\B$ in $V_p$ for $p\in \P,$   we consider the  locally compact 
group $GL(V_\AA),$ which is the  restricted product 
of the $GL(V_p)$'s with respect to the compact groups $GL(\B_p)$'s 
(see Subsection~\ref{SS-Dual}).

For $p\in \P \cup \{\infty\},$ let $G_p$ denote the closure of the image 
of $\Ga$ in $GL(V_p(\Ga))$.
Set 
 $$
 G:= (G_\infty\times \prod_{p\in \P} G_p)\cap GL(V_\AA).
 $$
\vskip.2cm 
$\bullet$ {\it Fifth step.} We claim that $G$ is a closed amenable subgroup of 
$GL(V_\AA).$

Indeed, for every $n\geq 1,$ set
$$
H_n:= G_\infty\times \prod_{i=1}^n G_{p_i} \times K_n,
$$
where $K_n$ is the compact group $\prod_{i>n} (G_{p_i}\cap GL(\B_{p_i}).$
Then $(H_n)_n$ is an increasing sequence of  open subgroups of $G$ 
and $G=\bigcup_{n\geq 1} H_n$. Clearly, every $H_n$ is a closed subgroup 
of $GL(V_\AA).$ Hence, $G$ is a locally compact and therefore  a closed subgroup of $GL(V_\AA).$

 To show that $G$ is amenable, it suffices to show that every $H_n$ is amenable (see \cite[Proposition G.2.2]{BHV}).
 This is indeed the case, since every $G_p$ is amenable by definition of $V_p(\Ga)$ and since $K_n$ is 
 compact.
 
 \vskip.2cm 
For every $n\in \NN,$ denote by $W^n$ the  $\QQ$-linear span  of 
$$\widehat{X}\cap  V_\infty(\Ga)\cap \bigcap_{i=1}^n V_{p_i}(\Ga).$$

\vskip.2cm 
$\bullet$ {\it Sixth step.} We claim that there  exists $N\in \NN$ such that 
$W^n= W^N$ for every $n\geq N.$

Indeed, $(W^n)_{n\geq 1}$ is a decreasing sequence of linear subspaces of $V_\QQ$.
By the fourth step, we have $\dim_\QQ W^n>0$ for every $n\geq 1.$ Hence, there exists $N\in \NN$ such that 
$\dim_\QQ W^n= \dim_\QQ W^N$ for every $n\geq N$ and the claim is proved.

 \vskip.2cm 
 Set $W:= W^N$ and observe that $W$ is $\Ga$-invariant.
 \vskip.2cm 
$\bullet$ {\it Seventh step.} We claim that the image of $\Ga$ in $\Aut(\widehat{X}\cap W)$ is amenable. 

Indeed,  $W$ is a subspace of $V_\QQ$ and is contained in every $V_p(\Ga).$
On the one hand,  under the  diagonal embedding, $G\cap GL(W)$ is a discrete subgroup of $G,$ since the neighbourhood 
$$U \times \prod_{p\in P} (G_{p}\cap GL(\B_{p}))$$
of $e$ in  $G$ has trivial intersection with $GL(W),$ for a sufficiently small neighbourhood $U$
of $e$ in $G_\infty.$
On the other hand, $G\cap GL(W)$ is amenable, by the fifth step.
It follows that  the image $\widetilde{\Ga}\subset G\cap GL(W)$ of $\Ga$ in $GL(W)$  is amenable. 
The image of $\Ga$ in $\Aut(\widehat{X}\cap W)$ is a quotient of $\widetilde{\Ga}$
and  is therefore amenable.

\vskip.2cm 
Let
$$Y:= (\widehat{X}\cap W)^\perp= \left\{x\in X\mid \chi(x)=1 \tout \chi\in \widehat{X}\cap W\right\} $$ 
be the annihilator in $X$ of  the subgroup $\widehat{X}\cap W$ of  $ \widehat{X}.$

\vskip.2cm 
$\bullet$ {\it Eighth step.} We claim $Y$ is a  $\Ga$-invariant  proper subsolenoid   of $X$  and  that the image of $\Ga$ in $\Aut(X/Y)$  is amenable.

Indeed, $Y$ is clearly a closed  $\Ga$-invariant subgroup of $X$ and $Y\neq X$  since $\widehat{X}\cap W$ is non trivial, by the fourth step. 
Moreover, the dual group $\widehat{Y}$ of $Y,$ which is isomorphic to $\widehat{X}/(\widehat{X}\cap W)$,
is torsion free: if $\chi\in \widehat{X}$ is such that $n\chi\in W$ for some integer $n\geq 1,$
then $\chi\in W$, since $W$ is a $\QQ$-linear subspace.
As, obviously, $\widehat{Y}$ has finite rank, it follows that the compact group $Y$ is a solenoid.

By the seventh step,  the image of $\Ga$ in $\Aut(\widehat{X}\cap W)$ is amenable.
Since, $\Aut(X/Y)$ is isomorphic to $\Aut(\widehat{X}\cap W)$ by duality, it follows that the image of $\Ga$ in $\Aut(X/Y)$  is amenable. $\square$

%%MODIFICATION-Begin
\subsection{Proof of Theorem~\ref{Theo2}}
\label{ProofTheo2}
We only have to show that (i) implies (ii).
Set $X:= \SS_a^d$ for $a=p_1\dots p_r$
and let $\Ga$ be a subgroup of $\Aff(\SS_a^d).$ 
As in the proof of Theorem~\ref{Theo1}, we may assume that 
$\Ga\subset \Aut(X).$ 

Recall from Subsection~\ref{SS-DualSol}\,  that we may identify
 $\widehat{X}$  with the discrete subring $\ZZ[1/a]^d$ of 
 $$\AA_a^d=\RRR^d\times\QQ_{p_1}^d\times\cdots\times \QQ_{p_r}^d$$ and 
  $\Aut(\widehat{X})$ with  the discrete subgroup $GL_d(\ZZ[1/a])$
 of  $GL_d(\AA_a)$, with the dual action of $\ga\in  \Aut(X)$ on  $\AA_a^d$
 given by matrix transpose.

As  in the proof of Theorem~\ref{Theo1}, there exists a  $\Ga$-invariant mean $M$ on $\widehat{X}\setminus \{0\}.$
Let $W$ be a non zero $\QQ$-linear subspace of $V_\QQ= \widehat{X} \otimes_\ZZ \QQ$
of  minimal dimension with $M(W)=1.$
Then $W$ is $\Ga$-invariant, by  $\Ga$-invariance of $M$.
We claim that the image of $\Ga$ in $GL(W)$ is virtually abelian. 

Indeed, fix $p\in \{p_1,\dots, p_r\}\cup \{\infty\}$. Set 
$W_p= W\otimes_\QQ \QQ_p$ and let  $G_p$ be the closure of the image  of $\Ga$ in $GL(W_p).$
Let $\mu_p$ be the $G_p$-invariant probability measure on $\PP(W_p)$ 
which is the pushforward of $M$ under the map $W\setminus \{0\}\to \PP(W_p).$
 Then $\mu_p$  is  not supported on a proper projective subspace of $W_p:$ if $W'$ is a
 $\QQ_p$-linear subspace of $W_p$ with $\mu_p([W'])=1,$ where $[W']$ is the image of 
 $W'$ in $\PP(W_p),$ then $M(W'\cap W)=1$ and hence
$W'\cap W=W$, by minimality of $W$;  so $W'=W\otimes_\QQ \QQ_p=W_p.$ 

By Proposition~\ref{Pro-Furstenberg},  there exists therefore a finite index subgroup
$H_p$ of $G_p$ with a relatively  compact commutator subgroup $[H_p,H_p].$
 
 Set $G= G_\infty\times \prod_{i=1}^r G_{p_i}.$ 
As in the proof of Theorem~\ref{Theo1}, the image $\widetilde{\Ga}$ of $\Ga$ in $G$ is discrete.
Then $\widetilde{\Ga_0}:=\widetilde{\Ga}\cap \prod_{i=1}^r H_{p_i}$ is a subgroup of finite index in 
$\widetilde{\Ga}$ and its commutator $[\widetilde{\Ga_0},\widetilde{\Ga_0}]$ is finite.
Since $\widetilde{\Ga_0}\subset GL(W)$ is linear, it follows therefore from Lemma~\ref{Lem-LinearVirtually}
that $\widetilde{\Ga_0}$ and hence $\widetilde{\Ga}$ is virtually abelian. 
This concludes the proof of Theorem~\ref{Theo2}. $\square$
%%MODIFICATION-end

\begin{remark}
\label{Rem3}
Let $X$ be as in Theorem~\ref{Theo1} and  let $\Ga$ be finitely generated subgroup of $\Aut(X).$ 
We claim that there exists  finitely many   different primes $p_1,\cdots, p_r$ such that 
$\Ga$ is contained in the subgroup $\Aut(\widehat{X})_{\ZZ[1/a]}$ defined in Subsection~\ref{SS-Dual},
where $a=p_1\dots p_r.$ 

Indeed, let $\ga_1, \dots, \ga_n$ be a set generators of $\Ga.$ Let   $\B$ be a basis of $V_\QQ=\widehat{X}\otimes_\ZZ \QQ$ over $\QQ$ contained 
in $\widehat{X}$ . Then  every $\ga_i$ leaves invariant  the $\ZZ_p$-module $\B_p$  generated by 
$\B$ in $V_p=V_\QQ\otimes_\QQ \QQ_p$ for almost every prime $p$ and the claim follows.

Assume that Item (i) in  Theorem~\ref{Theo1} holds for the action of $\Ga$ on $X.$ 
The proof of Theorem~\ref{Theo2} shows that there exists a $\Ga$-invariant subspace $W$ of  $V_\QQ$ such that 
the image of $\Ga$ in $GL( W)$  is virtually abelian.
 Then  $Y=(\widehat{X}\cap W)^\perp$ is a subsolenoid in $X$ and the  image of $\Ga$ in $\Aut(X/Y)$  is virtually abelian.
   
\end{remark}

\subsection{Proof of Proposition~\ref{Pro-Ergodic}}
\label{SS-Pro-Ergodic}
Let $\Ga$ be a subgroup of $\Aff(X)$.

Assume that there exists a proper closed subgroup $Y$ such that 
 the image $\overline{\Ga}$ of $\Ga$ in $\Aff(X/Y)$  is  finite.
Since $X$ is compact and connected, $\overline{X}=X/Y$ is a non trivial compact connected 
group. It is easy to see that there exist two $\overline{\Ga}$-invariant non empty open subsets 
of $\overline{X}$ which are disjoint. The preimages $U$ and $V$ of  these sets in $X$ are $\Ga$-invariant 
non empty open subsets and are disjoint. 
Since the support of the  Haar measure 
$\mu$ of  $X$ coincides with $X$, we have $\mu(U)\neq 0$ and $\mu(V)\neq 0.$ 
Hence, $\Ga\curvearrowright  (X,\mu)$ is not ergodic.

Conversely, assume that $\Ga\curvearrowright  (X,\mu)$ is not ergodic.
Since $X$ is connected, $\widehat{X}$ is torsion free. As in the previous sections, we view $\widehat{X}$ as subgroup
of the (possibly infinite dimensional) $\QQ$-vector space  $V_\QQ=\widehat{X}\otimes_\ZZ \QQ$.
We realize the associated Koopman representation $\pi_X$ of $\Ga$ in $\ell^2(\widehat{X})$
as in Subsection~\ref{SS-Reduction}.
By non ergodicity of the action, there exists a non-zero $\Ga$-invariant vector $\xi\in \ell^2(\widehat{X}\setminus \{0\})$.
Thus, we have (see Formula $(*)$ from  Subsection~\ref{SS-Reduction}) 
$$
\chi(x)\xi(\chi^{\theta})=\xi(\chi)  \tout \xi\in \ell^2(\widehat X) , \chi\in \widehat{X},\leqno(**)
$$
for all $(\theta, x)\in \Ga.$

Set $\eta:=|\xi|.$ Then $\eta\neq 0$ and  Formula $(**)$  shows that  $\eta$ is $\paut(\Ga)$-invariant.
Let $\chi_0\in \widehat{X}\setminus \{0\}$ be such that $\eta(\chi_0)\neq 0.$ 
Since $\eta\in  \ell^2(\widehat{X})$ and since  $\eta\neq 0,$
it follows that the $\paut(\Ga)$-orbit is finite.

Let $W$ be the linear span of the  $\paut(\Ga)$-orbit of $\chi_0$ in $V_\QQ$
and  let $Y:= (\widehat{X} \cap  W)^\perp$ be the annihilator of $\widehat{X} \cap  W$  in  $X$.
Then $Y$ is a $\paut(\Ga)$-invariant closed subgroup of $X$ and $Y\neq X$ since $\chi_0\neq 0.$

Moreover, $\widehat{Y}\cong \widehat{X}/(\widehat{X}\cap W)$ is torsion free and hence $Y$ is connected: if $\chi\in \widehat{X}$ is such that $n\chi\in W$ for some integer $n\geq 1,$ then $\chi\in W$, since $W$ is a $\QQ$-linear subspace.

We claim that the image of $\Ga$ in $\Aff(X/Y)$ is a finite group. Indeed, since the $\paut(\Ga)$-orbit of $\chi_0$ is
finite, we can find a normal subgroup  $\Lambda$  in $\paut(\Ga)$ of finite index which fixes $\chi_0$.
Set 
$$\Ga_0:=\paut^{-1}(\Lambda) \cap \Ga.$$
 Then $\Ga_0$ is  a normal subgroup of finite index in $\Ga.$ 

 Let $\ga=(\theta, x)\in \Ga_0.$
Formula $(**)$ shows that $\chi(x)=1$ for every $\chi$ in the $\paut(\Ga)$-orbit of $\chi_0$ and hence for 
every $\chi\in \widehat{X} \cap  W.$ Using Formula $(*)$ from  Subsection~\ref{SS-Reduction}, it follows that
$\Ga_0$ acts trivially under the Koopman representation $\pi_{X/Y}$ on $\ell^2(\widehat{X/Y})=\ell^2( \widehat{X} \cap  W)$ associated to the  action  of $\Ga$ on $X/Y$. So, the image of $\Ga_0$ in $\Aff(X/Y)$ is trivial and therefore
the image of $\Ga$ in $\Aff(X/Y)$ is finite.$\square$

\end{document}